\documentclass[11pt,reqno]{amsart} 
\usepackage[left=1in, right=1in, top=1in, bottom=1in]{geometry}
\usepackage{amsmath, amsthm, amsfonts,color,url,enumitem}
\usepackage{cancel}
\usepackage{wrapfig,hyperref}
\usepackage{physics}
\usepackage{mathtools}
\usepackage{subfigure}
\usepackage{color}
\usepackage{thmtools}
\usepackage{thm-restate}
\usepackage{hyperref}
\usepackage{cleveref}
\usepackage[normalem]{ulem}
\usepackage{tikz}
\usepackage{dsfont}
\usepackage{physics}
\usepackage{floatrow}
\usepackage{url}
\usepackage{graphicx}
 \usepackage{diagbox}

\theoremstyle{plain}
\newtheorem{theorem}{Theorem}[section]
\newtheorem{proposition}[theorem]{Proposition}
\newtheorem{corollary}[theorem]{Corollary}

\theoremstyle{definition}

\newtheorem{remark}[theorem]{Remark}

\numberwithin{equation}{section}

\DeclareMathOperator{\Sym}{\mathfrak{S}}

\DeclareMathOperator{\PF}{PF}
\DeclareMathOperator{\PPF}{PPF}

\newcommand{\ER}{\mathbb{E}}
\newcommand{\PR}{\mathbb{P}}

\newtheorem{innercustomthm}{Theorem}
\newenvironment{customthm}[1]
  {\renewcommand\theinnercustomthm{#1}\innercustomthm}
  {\endinnercustomthm}

\title{Fixed Points and Cycles of Parking Functions}

\author[Rubey]{Martin Rubey}
\address[Rubey]{Fakult\"{a}t f\"{u}r Mathematik und Geoinformation, TU Wien, Vienna, Austria}
\email{\textcolor{blue}{\textcolor{blue}{\href{mailto:martin.rubey@tuwien.ac.at}{martin.rubey@tuwien.ac.at}}}}

\author[Yin]{Mei Yin}
\thanks{M.~Yin was supported by the University of Denver's Professional Research Opportunities for Faculty Fund 80369-145601 and Simons Foundation Grant MPS-TSM-00007227.}
\address[Yin]{Department of Mathematics, University of Denver, Denver, CO, USA}
\email{\textcolor{blue}{\textcolor{blue}{\href{mailto:mei.yin@du.edu}{mei.yin@du.edu}}}}

\begin{document}



\begin{abstract}
A \emph{parking function} of length $n$ is a sequence $\pi=(\pi_1,\dots, \pi_n)$ 
of positive integers such that if $\lambda_1\leq\cdots\leq \lambda_n$ is the increasing
rearrangement of $\pi_1,\dots,\pi_n$, then $\lambda_i\leq i$ for $1\leq i\leq n$. The index $i$ is a fixed point of the parking function $\pi$ if $\pi_i=i$. More generally, for $m\geq 1$, the indices $(i_1, \dots, i_m)$ where the $i_j$'s are all distinct constitute an $m$-cycle of the parking function $\pi$ if $\pi_{i_1}=i_2, \pi_{i_2}=i_3, \dots, \pi_{i_{m-1}}=i_m, \pi_{i_m}=i_1$.
In this paper we obtain some exact results on the number of fixed points and cycles of parking functions. Our derivations are based on generalizations of Pollak's argument and the symmetry of parking coordinates. Extensions of our techniques are discussed.
\end{abstract}

\maketitle

\section{Introduction}
\label{sec:intro}
Parking functions were introduced by Konheim and Weiss \cite{k-w} in the study of the linear probes
of random hashing functions. In the classical parking
function scenario, we have $n$
parking spaces on a one-way street, labelled $1,2,\dots,n$ in
consecutive order as we drive down the street. There are $n$ cars
$C_1,\dots,C_n$. Each car $C_i$ has a preferred space $1\leq \pi_i\leq
n$. The cars drive down the street one at a time in order
$C_1,\dots, C_n$. The car $C_i$ drives immediately to space $\pi_i$ and
then parks in the first available space. Thus if $\pi_i$ is empty, then
$C_i$ parks there; otherwise $C_i$ next goes to space $\pi_i+1$, $\pi_i+2$, etc., until it finds an available space to park in (if no such space exists, then $C_i$ leaves the street unparked). If
all cars are able to park, then the sequence $\pi=(\pi_1,\dots, \pi_n)$ is
called a \emph{parking function} of length $n$. It is well-known
and easy to see that if $\lambda_1\leq\cdots\leq \lambda_n$ is the (weakly) increasing
rearrangement of $\pi_1,\dots,\pi_n$, then $\pi$ is a parking function if
and only if $\lambda_i\leq i$ for $1\leq i\leq n$. Equivalently, $\pi$ is a parking function if and only if
\begin{equation*}
\#\{k: \pi_k \leq i\} \geq i, \hspace{.2cm} \forall i=1, \dots, n.
\end{equation*}
This implies that parking functions are invariant under the action of the symmetric group $\Sym_n$ permuting the $n$ cars, that is, permuting the list of preferences $\pi$. Write $\PF_n$
for the set of parking functions of length $n$.

The first significant result on parking functions, due to Pyke
\cite{pyke} in another context and then to Konheim and Weiss
\cite{k-w}, is that the number of parking functions of length $n$ is
equal to $(n+1)^{n-1}$. A famous combinatorial proof was given by
Pollak (unpublished but recounted in \cite{Pollak} and
\cite{Pollak2}). It boils down to the following easily verified
statement: Let $G$ denote the group of all $n$-tuples
$(a_1,\dots,a_n)\in [n+1]^n$ with componentwise addition modulo
$n+1$ (where $[n]$ is the standard shorthand for $\{1, \dots, n\}$). Let $H$ be the subgroup generated by $(1,1,\dots,1)$. Then
every coset of $H$ contains exactly one parking function.

Nowadays the study of parking functions has found many applications in combinatorics, probability,
computer science, and beyond. Some recent work may be found in \cite{Adeniran, Colmenarejo, Ji, Kang, Yan2}. We refer to Yan \cite{Yan} for a comprehensive survey and to Carlson et al. \cite{carlson2020parking} for the many generalizations of the classical parking function scenario. In this paper we explore the cycle structure of parking functions. Unlike the cycle structure of permutations which has been well studied, not much is known about the cycle structure of parking functions, especially in terms of exact enumerative results. A traditional direction is to standardize a parking function to a permutation and then take the usual cycle structure. This was discussed in detail in Stanley \cite{Stanley} supplementary problems for Chapter 7, utilizing the powerful parking function symmetric function machinery. Our investigation makes an effort in another direction. See Diaconis and Hicks \cite{DH} for some open problems and Paguyo \cite{Paguyo} for some partial answers in this direction. For a parking function $\pi=(\pi_1,\dots, \pi_n)$, we say $i$ is a fixed point of $\pi$ if $\pi_i=i$. More generally, for $m\geq 1$, we say $(i_1, \dots, i_m)$ where the $i_j$'s are all distinct is an $m$-cycle of $\pi$ if $\pi_{i_1}=i_2, \pi_{i_2}=i_3, \dots, \pi_{i_{m-1}}=i_m, \pi_{i_m}=i_1$. Fixed points are simply $1$-cycles and cyclic points are points that lie on some $m$-cycle of $\pi$ where $m \geq 1$.

We are interested in deriving some exact formulas for the number of fixed points and cycles of parking functions. Our counting formulas are based on generalizations of Pollak's argument and the symmetry of parking coordinates. In Section \ref{sec:classical} we study classical parking functions, in Section \ref{sec:prime} we study prime parking functions, and in Section \ref{sec:dis} we discuss extensions of our techniques to $(r, k)$-parking functions (see Section \ref{sec:gen}) and present connections of our current investigation to other research areas through OEIS entries (see Section \ref{sec:connect}). While some OEIS entries are easy to explain, many entries remain mysterious, waiting to be explored deeper in the future.

Theorems \ref{classical-fixed} and \ref{classical-gen} come as a particular surprise:

\begin{customthm}{2.1}\label{classical-fixed}
Let $0\leq k\leq n$. The number of parking functions of length $n$ with $k$ fixed points is given by
\begin{equation*}
\frac{1}{(n+1)^2} \binom{n+1}{k} \left(n^{n-k+1}-(-1)^{n-k+1}\right).
\end{equation*}
\end{customthm}

\begin{customthm}{2.2}\label{classical-gen}
The generating function for parking functions of length $n$ with respect to the number of fixed points is given by
\begin{equation*}
\frac{1}{(n+1)^2} \left((q+n)^{n+1}-(q-1)^{n+1}\right).
\end{equation*}
\end{customthm}

In contrast, there is not such a nice and clean formula for permutations, which despite their simpler structure are characterized by the intricate generating function
\begin{equation*}
\sum_{n=0}^\infty \sum_{k=0}^n \frac{F(k, n)z^n q^k}{n!}=\frac{\exp((q-1)z)}{1-z},
\end{equation*}
where $F(k, n)$ is the number of permutations of $[n]$ with $k$ fixed points. 

Theorems \ref{prime-fixed} and \ref{prime-gen} are the corresponding prime versions for parking functions:

\begin{customthm}{3.1}\label{prime-fixed}
Let $0\leq k\leq n-1$. The number of prime parking functions of length $n$ with $k$ fixed points is given by
\begin{equation*}
\binom{n-1}{k}(n-2)^{n-k-1}.
\end{equation*}
\end{customthm}

\begin{customthm}{3.2}\label{prime-gen}
The generating function for prime parking functions of length $n$ with respect to the number of fixed points is given by $(q+n-2)^{n-1}$.
\end{customthm}

Extending the notion of fixed points to cycles, we obtain more general results for cycles in classical parking functions in Theorem \ref{classical-cycle} and prime parking functions in Theorem \ref{prime-cycle}. Our exact formulas for fixed points and cycles furthermore provide insight into the asymptotic features of the cycle structure of random parking functions, particularly in connection with random functions and random permutations whose cycle structures have already been quite well understood.

\begin{itemize}

\item A random parking function is a random function in $\mathcal{F}_n=\{f: [n] \rightarrow [n]\}$ conditioned on being in $\PF_n$.

\item At the other end of the spectrum, we have random permutations, which are special cases of random parking functions where all cars have different preferred spaces.
\end{itemize}

As stated by Diaconis and Hicks \cite{DH}, in analogy with \emph{equivalence of ensembles} in statistical mechanics, it is natural to expect that for some features, the distribution of the features in the ``micro-canonical ensemble'' ($\PF_n$) should be close to the features in the ``canonical ensemble'' ($\mathcal{F}_n$). We will show in Remark \ref{classical-remark} that the expected number of $m$-cycles in a random parking function is asymptotically $1/m$ for any $m$. This coincides with the corresponding asymptotic result for random functions as well as random permutations. Nevertheless error terms from $1/m$ are not the same or uniform for the three different random objects.

\begin{itemize}

\item For random permutations in the symmetric group $\Sym_n$, the asymptotic formula $1/m$ for the expected number of $m$-cycles is also exact. Based on this fact, Shepp and Lloyd \cite{SL} employed a generating function approach to show that the total number of cycles in a uniformly random permutation from $\Sym_n$ is asymptotically normal with mean and variance $\log n$.

\item For random functions in $\mathcal{F}_n$, the exact formula for the expected number of $m$-cycles is $(n)_m / (m n^m)$, where $(n)_m$ is a falling factorial. Building upon this, Flajolet and Odlyzko \cite{FO} applied generating function techniques to show that the total number of cycles in a uniformly random function from $\mathcal{F}_n$ is asymptotically normal with mean and variance $\frac{1}{2} \log n$. Note that asymptotically, random functions have about half as many total cycles as random permutations.

\item Now comes our contribution. In Theorem \ref{classical-average}, we present the exact formula for the expected number of $m$-cycles in random parking functions in $\PF_n$, and it should be tractable to use this result to check against the equivalence of ensembles heuristic. Theorem \ref{prime-average} provides the corresponding prime version.

\end{itemize}

\section{Classical parking functions}
\label{sec:classical}
Our first result identifies the number of parking functions with a prescribed number of fixed points.
\begin{theorem}\label{classical-fixed}
Let $0\leq k\leq n$. The number of parking functions of length $n$ with $k$ fixed points is given by
\begin{equation*}
\frac{1}{(n+1)^2} \binom{n+1}{k} \left(n^{n-k+1}-(-1)^{n-k+1}\right).
\end{equation*}
\end{theorem}

\begin{proof}
For $1\leq i\leq n$, let $A_i=\vert\{\pi \in \PF_n: \pi_i=i\} \vert$. Recall that the number of parking functions of length $n$ is $\vert \PF_n \vert=(n+1)^{n-1}$. For temporarily fixed $1\leq i_1<\cdots<i_k\leq n$, by De Morgan's law and the inclusion-exclusion principle,
\begin{align*}
&\left \vert \bigcap_{i \neq i_j \, \forall 1\leq j\leq k} A_i^c \right \vert = \left \vert \PF_n - \bigcup_{i \neq i_j \, \forall 1\leq j\leq k} A_i \right \vert \\
=& \vert \PF_n \vert - \sum_{\ell=1}^{n-k} (-1)^{\ell+1}  \sum_{\substack{1\leq s_1<\cdots<s_\ell\leq n \\ s_t \neq i_j \, \forall 1\leq t\leq \ell\, \forall 1\leq j\leq k}} \left \vert \bigcap_{t=1}^{\ell} A_{s_t} \right \vert \\
=&\sum_{\ell=0}^{n-k} (-1)^{\ell}  \sum_{\substack{1\leq s_1<\cdots<s_\ell\leq n \\ s_t \neq i_j \, \forall 1\leq t\leq \ell\, \forall 1\leq j\leq k}} \left \vert \bigcap_{t=1}^{\ell} A_{s_t} \right \vert.
\end{align*}
The number of parking functions of length $n$ with $k$ fixed points is thus
\begin{align*}
&\hspace{2cm}\sum_{1\leq i_1<\cdots<i_k\leq n}\left\vert \left(\bigcap_{j=1}^{k} A_{i_j}\right) \bigcap \left(\bigcap_{i \neq i_j \, \forall 1\leq j\leq k} A_i^c\right)\right\vert\\=&\sum_{1\leq i_1<\cdots<i_k\leq n} \sum_{\ell=0}^{n-k} (-1)^\ell  \sum_{\substack{1\leq s_1<\cdots<s_\ell\leq n \\ s_t \neq i_j \, \forall 1\leq t\leq \ell\, \forall 1\leq j\leq k}} \left\vert \left(\bigcap_{j=1}^{k} A_{i_j}\right) \bigcap \left(\bigcap_{t=1}^{\ell} A_{s_t}\right) \right\vert \\
=&\sum_{\ell=0}^{n-k} (-1)^\ell \sum_{1\leq i_1<\cdots<i_k\leq n} \sum_{\substack{1\leq s_1<\cdots<s_\ell\leq n \\ s_t \neq i_j \, \forall 1\leq t\leq \ell\, \forall 1\leq j\leq k}} \left\vert \left(\bigcap_{j=1}^{k} A_{i_j}\right) \bigcap \left(\bigcap_{t=1}^{\ell} A_{s_t}\right) \right\vert\\
=&\sum_{\ell=0}^{n-k} (-1)^\ell \left\vert \left\{\pi \in \PF_n: \substack{\pi_1<\cdots<\pi_{k}, \pi_{k+1}<\cdots<\pi_{k+\ell} \\ \pi_1, \ldots, \pi_{k+\ell} \text{ are pairwise distinct}} \right\} \right\vert,
\end{align*}
where the last equality follows by symmetry of parking coordinates, since any permutation of a parking function is a parking function. 

For ease of notation, let $$A_{k+\ell}=\left\{\pi \in \PF_n: \substack{\pi_1<\cdots<\pi_{k}, \pi_{k+1}<\cdots<\pi_{k+\ell} \\ \pi_1, \ldots, \pi_{k+\ell} \text{ are pairwise distinct}} \right\}.$$
Further let $$A'_{k+\ell}=\left\{\pi \in \PF_n: \pi_1, \pi_2, \dots, \pi_{k+\ell} \text{ are pairwise distinct} \right\},$$ $$B'_{k+\ell}=\left\{f: [n] \rightarrow [n+1]: f(1), f(2), \dots, f(k+\ell) \text{ are pairwise distinct}\right\}.$$
It is clear that $\vert A'_{k+\ell}\vert=k! \ell! \vert A_{k+\ell}\vert$. To obtain a simplified expression for $A_{k+\ell}$, we use an extension of Pollak's circle argument. Add an additional parking spot $n+1$, and arrange the spots in a circle. Allow $n+1$ also as a preferred spot. We first select $k+\ell$ spots for the first $k+\ell$ cars, which can be done in $\binom{n+1}{k+\ell}$ ways. Then for the remaining $n-k-\ell$ cars, there are $(n+1)^{n-k-\ell}$ possible preference sequences. Hence
\begin{equation*}
\vert B'_{k+\ell}\vert=(k+\ell)! \binom{n+1}{k+\ell}(n+1)^{n-k-\ell}.
\end{equation*}
Out of the $n+1$ rotations for any preference sequence, only one rotation becomes a valid parking function. Standard circular symmetry argument yields
\begin{equation*}
\vert A'_{k+\ell}\vert=\frac{1}{n+1} \vert B'_{k+\ell}\vert=(k+\ell)! \binom{n+1}{k+\ell}(n+1)^{n-k-\ell-1},
\end{equation*}
and so \begin{equation*}
\vert A_{k+\ell} \vert = \binom{n+1}{k, \ell, n-k-\ell+1} (n+1)^{n-k-\ell-1}.
\end{equation*}

So what we need to compute is
\begin{align*}
&\sum_{\ell=0}^{n-k} (-1)^{\ell} \binom{n+1}{k, \ell, n-k-\ell+1} (n+1)^{n-k-\ell-1} \\
=&\binom{n+1}{k} \sum_{\ell=0}^{n-k} (-1)^\ell \binom{n-k+1}{\ell} (n+1)^{n-k-\ell-1} \\
=&\binom{n+1}{k}\frac{1}{(n+1)^2} \left(\sum_{\ell=0}^{n-k+1} (-1)^\ell \binom{n-k+1}{\ell} (n+1)^{n-k+1-\ell}-(-1)^{n-k+1} \right) \\
=&\binom{n+1}{k}\frac{((n+1)-1)^{n-k+1}-(-1)^{n-k+1}}{(n+1)^2}.
\end{align*}
\end{proof}

\begin{theorem}\label{classical-gen}
The generating function for parking functions of length $n$ with respect to the number of fixed points is given by
\begin{equation*}
\frac{1}{(n+1)^2} \left((q+n)^{n+1}-(q-1)^{n+1}\right).
\end{equation*}
\end{theorem}

\begin{proof}
Utilizing Theorem \ref{classical-fixed}, we perform generating function calculations.
\begin{align*}
&\sum_{k=0}^n q^k \frac{1}{(n+1)^2} \binom{n+1}{k} \left(n^{n-k+1}-(-1)^{n-k+1}\right) \notag \\
=&\frac{1}{(n+1)^2} \left(\sum_{k=0}^{n+1} \binom{n+1}{k} q^k n^{n-k+1}-\sum_{k=0}^{n+1} \binom{n+1}{k} q^k (-1)^{n-k+1}\right) \notag \\
=&\frac{1}{(n+1)^2} \left((q+n)^{n+1}-(q-1)^{n+1}\right).
\end{align*}
\end{proof}

Extending the notion of fixed points to cycles, we have a more general result.

\begin{theorem}\label{classical-cycle}
Let $m\geq 1$ and $k\geq 0$ with $km \leq n$. The number of parking functions of length $n$ with $k$ $m$-cycles is given by
\begin{equation*}
\sum_{\ell: (k+\ell)m \leq n} (-1)^{\ell} \frac{\left((m-1)!\right)^{k+\ell}}{k!\ell!} \binom{n+1}{\underbracket[0.5pt]{m, \cdots, m}_{k+\ell \hspace{.1cm} \text{$m$'s}}, n-(k+\ell)m+1} (n+1)^{n-(k+\ell)m-1}.
\end{equation*}
\end{theorem}

\begin{proof}
The proof follows similarly as in Theorem \ref{classical-fixed}. We will not include all the technical details, but instead walk through the key ideas.

For distinct $1\leq i_1, \dots, i_m\leq n$, let $$A_{(i_1, \dots, i_m)}=\vert\{\pi \in \PF_n: \pi_{i_1}=i_2, \pi_{i_2}=i_3, \dots, \pi_{i_{m-1}}=i_m, \pi_{i_m}=i_1\} \vert.$$
We still apply De Morgan's law and the inclusion-exclusion principle. In analogy with Theorem \ref{classical-fixed}, we need to compute
\begin{align*}
\sum_{\ell: (k+\ell)m \leq n} (-1)^{\ell} \left((m-1)!\right)^{k+\ell} \sum_{(i^j_1, \dots, i^j_m): 1\leq j\leq k} \sum_{(s^t_1, \dots, s^t_m): 1\leq t\leq \ell} \left \vert \left(\bigcap_{j=1}^k A_{(i^j_1, \dots, i^j_m)} \right) \bigcap \left(\bigcap_{t=1}^\ell A_{(s^t_1, \dots, s^t_m)} \right) \right \vert,
\end{align*}
where the middle sum is over all possible $k$ $m$-cycles and the last sum is over all possible $\ell$ $m$-cycles, and all the $k+\ell$ $m$-cycles are pairwise non-overlapping. To avoid overcounting, we arrange the $m$-cycles $(i^j_1, \dots, i^j_m)$ so that for each $1\leq j\leq k$, $i^j_1<\cdots<i^j_m$ and further $i^1_1<\cdots<i^k_1$, and similarly for each $1\leq t\leq \ell$, $s^t_1<\cdots<s^t_m$ and further $s^1_1<\cdots<s^\ell_1$. This hence introduces an additional scalar factor $(m-1)!$ for each such $m$-cycle, which accounts for the number of ways of arranging an $m$-cycle with distinct entries.

Next we use symmetry of parking coordinates and apply an extension of Pollak's circle argument. The major difference from Theorem \ref{classical-fixed} is that instead of selecting $k+\ell$ spots on a circle with $n+1$ spots, we select $k+\ell$ $m$-cycles on the circle. Again in analogy with Theorem \ref{classical-fixed},
\begin{align*}
&\sum_{(i^j_1, \dots, i^j_m): 1\leq j\leq k} \sum_{(s^t_1, \dots, s^t_m): 1\leq t\leq \ell} \left \vert \left(\bigcap_{j=1}^k A_{(i^j_1, \dots, i^j_m)} \right) \bigcap \left(\bigcap_{t=1}^\ell A_{(s^t_1, \dots, s^t_m)} \right) \right \vert \\
=&\left\vert\left\{\pi \in \PF_n: \substack{\pi^1_1<\cdots<\pi^k_1, \pi^{k+1}_1<\cdots<\pi^{k+\ell}_1 \\
(\pi^1_1, \dots, \pi^1_m), \dots, (\pi^{(k+\ell)}_1, \dots, \pi^{(k+\ell)}_m) \text{ are pairwise non-overlapping}} \right\} \right\vert \\
=&\frac{1}{k!\ell!}\left\vert\left\{\pi \in \PF_n: (\pi^1_1, \dots, \pi^1_m), \dots, (\pi^{(k+\ell)}_1, \dots, \pi^{(k+\ell)}_m) \text{ are pairwise non-overlapping} \right\} \right\vert \\
=&\frac{1}{k!\ell!} \binom{n+1}{\underbracket[0.5pt]{m, \cdots, m}_{k+\ell \hspace{.1cm} \text{$m$'s}}, n-(k+\ell)m+1} (n+1)^{n-(k+\ell)m-1}.
\end{align*}
Our conclusion follows.
\end{proof}

As is standard, we denote the probability distribution and expectation with respect to the uniform measure on the set of classical parking functions of length $n$ by $\PR_n$ and $\ER_n$, respectively. The following theorem provides an exact result on the expected number of $m$-cycles in a classical parking function drawn uniformly at random.

\begin{theorem}\label{classical-average}
Take $1\leq m\leq n$. Let $\pi \in \PF_n$ be a parking function chosen uniformly at random and $C_m(\pi)$ be the number of $m$-cycles of $\pi$. The expected number of $m$-cycles is given by
\begin{equation*}
\ER_n(C_m(\pi))=(m-1)!\binom{n+1}{m}/(n+1)^{m}.
\end{equation*}
\end{theorem}

\begin{remark}\label{classical-remark}
From Theorem \ref{classical-average}, we may readily derive that for fixed $m$, as $n$ gets large, the expected number of $m$-cycles in a random parking function is asymptotically $1/m$. This asymptotic result was obtained earlier by Paguyo \cite{Paguyo} using other techniques.
\end{remark}

\begin{proof}
For distinct $1\leq i_1, \dots, i_m\leq n$, let $$A_{(i_1, \dots, i_m)}=\vert\{\pi \in \PF_n: \pi_{i_1}=i_2, \pi_{i_2}=i_3, \dots, \pi_{i_{m-1}}=i_m, \pi_{i_m}=i_1\} \vert.$$
By linearity of expectation and symmetry of parking coordinates,
\begin{align*}
\ER_n(C_m(\pi))&=\sum_{1\leq i_1<\cdots<i_m\leq n} (m-1)! \, \PR_n(A_{(i_1, \dots, i_m)}) \\
&=\sum_{1\leq i_1<\cdots<i_m\leq n} (m-1)! \, \PR_n(\pi_{1}=i_1, \dots, \pi_{m}=i_m)\\
&=\frac{(m-1)!}{(n+1)^{n-1}} \, \left\vert \left\{\pi \in \PF_n: \pi_1<\pi_2<\cdots<\pi_m \right\}\right\vert,
\end{align*}
where the additional scalar factor $(m-1)!$ in the first equation accounts for the number of ways of arranging an $m$-cycle with distinct entries $i_1, \dots, i_m$.

For ease of notation, let $$A_m=\left\{\pi \in \PF_n: \pi_1<\pi_2<\cdots<\pi_m \right\}.$$
Further let $$A'_m=\left\{\pi \in \PF_n: \pi_1, \pi_2, \dots, \pi_m \text{ are pairwise distinct} \right\}$$ and $$B'_m=\left\{f: [n] \rightarrow [n+1]: f(1), f(2), \dots, f(m) \text{ are pairwise distinct}\right\}.$$
It is clear that $\vert A'_m\vert=m!\vert A_m\vert$. To obtain a simplified expression for $A_m$, we use an extension of Pollak's circle argument as in the proof of Theorem \ref{classical-fixed}. Add an additional parking spot $n+1$, and arrange the spots in a circle. Allow $n+1$ also as a preferred spot. We first select $m$ spots for the first $m$ cars, which can be done in $\binom{n+1}{m}$ ways. Then for the remaining $n-m$ cars, there are $(n+1)^{n-m}$ possible preference sequences. Hence
\begin{equation*}
\vert B'_m\vert=m! \binom{n+1}{m}(n+1)^{n-m}.
\end{equation*}
Out of the $n+1$ rotations for any preference sequence, only one rotation becomes a valid parking function. Standard circular symmetry argument yields
\begin{equation*}
\vert A'_m\vert=\frac{1}{n+1} \vert B'_m\vert=m! \binom{n+1}{m}(n+1)^{n-m-1},
\end{equation*}
and so \begin{equation*}
\vert A_m \vert = \binom{n+1}{m}(n+1)^{n-m-1}.
\end{equation*}
Our conclusion follows.
\end{proof}

\begin{corollary}\label{classical-cor}
Let $\pi \in \PF_n$ be a parking function chosen uniformly at random. The expected number of cyclic points is given by
\begin{equation*}
\sum_{m=1}^n m!\binom{n+1}{m}/(n+1)^m.
\end{equation*}
\end{corollary}

\begin{proof}
As stated in the introduction, cyclic points are points that lie on some $m$-cycle of $\pi$ where $m \geq 1$. We apply Theorem \ref{classical-average} and note that every $m$-cycle contains $m$ cyclic points.
\end{proof}

\section{Prime parking functions}
\label{sec:prime}
A classical parking function
$\pi=(\pi_1, \dots, \pi_n)$ is said to be prime if for all $1 \leq j
\leq n-1$, at least $j+1$ cars want to park in the first $j$
places. (Equivalently, if we remove some term of $\pi$ equal to 1,
then we still have a parking function.) Denote the set of prime
parking functions of length $n$ by $\PPF_n$. Note that for prime parking functions it is impossible to have $n$ fixed points or to have an $n$-cycle.

As with classical
parking functions, we can also study prime parking functions via
circular rotation. The modified circular symmetry argument was first due to Kalikow \cite[pp.~141-142]{Stanley}, who provided the following
observation: Let $G$ denote the group of all $n$-tuples
$(a_1,\dots,a_n)\in [n-1]^n$ with componentwise addition modulo
$n-1$. Let $H$ be the subgroup generated by $(1,1,\dots,1)$. Then
every coset of $H$ contains exactly one prime parking function. The results we present in this section are thus largely parallel to the results for classical parking functions in our previous section.

\begin{theorem}\label{prime-fixed}
Let $0\leq k\leq n-1$. The number of prime parking functions of length $n$ with $k$ fixed points is given by
\begin{equation*}
\binom{n-1}{k}(n-2)^{n-k-1}.
\end{equation*}
\end{theorem}

\begin{proof}
We proceed as in the proof of Theorem \ref{classical-fixed}. The only difference is that the circular symmetry argument is now for a circle with $n-1$ spots. So we have
\begin{align*}
&\sum_{\ell=0}^{n-k-1} (-1)^{\ell} \binom{n-1}{k, \ell, n-k-\ell-1} (n-1)^{n-k-\ell-1} \\
=& \binom{n-1}{k} \sum_{\ell=0}^{n-k-1} (-1)^\ell \binom{n-k-1}{\ell} (n-1)^{n-k-1-\ell} \\
=&\binom{n-1}{k}(n-2)^{n-k-1}. \hspace{2cm} \qedhere
\end{align*}
\end{proof}

\begin{theorem}\label{prime-gen}
The generating function for prime parking functions of length $n$ with respect to the number of fixed points is given by $(q+n-2)^{n-1}$.
\end{theorem}

\begin{proof}
Utilizing Theorem \ref{prime-fixed}, we perform generating function calculations.
\begin{equation*}
\sum_{k=0}^{n-1} q^k \binom{n-1}{k}(n-2)^{n-k-1}=(q+n-2)^{n-1}.
\end{equation*}
\end{proof}

As in the case of classical parking functions, a more general theorem exists for prime parking functions when we extend the notion of fixed points to cycles.

\begin{theorem}\label{prime-cycle}
Let $m\geq 1$ and $k\geq 0$ with $km \leq n-1$. The number of prime parking functions of length $n$ with $k$ $m$-cycles is given by
\begin{equation*}
\sum_{\ell: (k+\ell)m \leq n-1} (-1)^{\ell} \frac{\left((m-1)!\right)^{k+\ell}}{k!\ell!} \binom{n-1}{\underbracket[0.5pt]{m, \cdots, m}_{k+\ell \hspace{.1cm} \text{$m$'s}}, n-(k+\ell)m-1} (n-1)^{n-(k+\ell)m-1}.
\end{equation*}
\end{theorem}

\begin{proof}
We proceed as in the proof of Theorem \ref{classical-cycle} and note that the circular symmetry argument is now for a circle with $n-1$ spots.
\end{proof}

As is standard, we denote the probability distribution and expectation with respect to the uniform measure on the set of prime parking functions of length $n$ by $\PR_n$ and $\ER_n$, respectively. The following theorem provides an exact result on the expected number of $m$-cycles in a prime parking function drawn uniformly at random.

\begin{theorem}\label{prime-average}
Take $1\leq m\leq n-1$. Let $\pi \in \PPF_n$ be a prime parking function chosen uniformly at random and $C_m(\pi)$ be the number of $m$-cycles of $\pi$. The expected number of $m$-cycles is given by
\begin{equation*}
\ER_n(C_m(\pi))=(m-1)!\binom{n-1}{m}/(n-1)^{m}.
\end{equation*}
\end{theorem}

\begin{remark}
From Theorem \ref{prime-average}, we may readily derive that for fixed $m$, as $n$ gets large, the expected number of $m$-cycles in a random prime parking function is asymptotically $1/m$.
\end{remark}

\begin{proof}
We proceed as in the proof of Theorem \ref{classical-average}. The only difference is that the circular symmetry argument is now for a circle with $n-1$ spots. So we have
$$\ER_n(C_m(\pi))=\frac{(m-1)!}{(n-1)^{n-1}} \binom{n-1}{m} (n-1)^{n-m-1}=(m-1)! \binom{n-1}{m} / (n-1)^{m}.$$
\end{proof}

\begin{corollary}
Let $\pi \in \PPF_n$ be a prime parking function chosen uniformly at random. The expected number of cyclic points is given by
\begin{equation*}
\sum_{m=1}^{n-1} m!\binom{n-1}{m}/(n-1)^m.
\end{equation*}
\end{corollary}

\begin{proof}
As in the proof of Corollary \ref{classical-cor}, we apply Theorem \ref{prime-average} and note that every $m$-cycle contains $m$ cyclic points.
\end{proof}

\section{Further discussions}
\label{sec:dis}
\subsection{Extensions of our techniques}
\label{sec:gen}
Fix some positive integers $k$ and $r$. Define an $(r, k)$-parking function of length $n$ \cite{SW} to be a sequence $(\pi_1, \dots, \pi_n)$ of positive integers such that if $\lambda_1\leq \cdots \leq \lambda_n$ is the (weakly) increasing
rearrangement of $\pi_1,\dots,\pi_n$, then $\lambda_i \leq k+(i-1)r$ for $1\leq i\leq n$. 
Denote the set of $(r, k)$-parking functions of length $n$ by
$\PF_n(r, k)$. There is a similar interpretation
for such $(r, k)$-parking functions in terms of the classical parking function
scenario: One wishes to park $n$ cars on a street with $k+(n-1)r$ spots,
but only $n$ spots are still empty, which are at positions no later
than $k, k+r, \dots, k+(n-1)r$. 

Pollak's original circle argument \cite{Pollak} may be further extended to $(r, k)$-parking functions.
Let $G$ denote the group of all $n$-tuples
$(a_1,\dots,a_n)\in [k+nr]^n$ with componentwise addition modulo
$k+nr$. Let $H$ be the subgroup generated by $(1,1,\dots,1)$. Then
every coset of $H$ contains exactly $k$ $(r, k)$-parking functions. See \cite{Stanley1} for more details.

The following proposition is then immediate.

\begin{proposition}\label{decomposition}
Let $0\leq s \leq n$. We have
\begin{equation*}
\vert \{\pi \in \PF_n(r, k): \pi_1<\pi_2<\dots<\pi_s\} \vert =k\binom{k+nr}{s}(k+nr)^{n-s-1}.
\end{equation*}
\end{proposition}

\begin{proof}
We apply the generalized circle argument as described above. Arrange $k+nr$ spots in a circle. We first select $s$ spots for the first $s$ cars, which can be done in $\binom{k+nr}{s}$ ways. Then for the remaining $n-s$ cars, there are $(k+nr)^{n-s}$ possible preference sequences. Out of the $k+nr$ rotations for any preference sequence, exactly $k$ rotations become valid $(r, k)$-parking functions. The rest is a standard circular symmetry argument as in the proof of Theorems \ref{classical-fixed} and \ref{classical-average}.
\end{proof}

\subsection{Connections to other research areas}
\label{sec:connect}
By setting $k=0, \dots, n$ in Theorem \ref{classical-fixed}, a plethora of findings emerge, particularly through OEIS entries.

\begin{itemize}
    \item For $k=0$, our sequence is included as \cite[St001903]{FindStat} with explanation and coincides with \cite[A081215]{OEIS}.
    \item For $k=1$, our sequence coincides with \cite[A081216]{OEIS} and is related to the dimension of the primitive middle cohomology of Dwork hypersurfaces.
    \item For $k=n-2$, our sequence coincides with \cite[A006325]{OEIS} which is the $4$-dimensional analog of centered polygonal numbers.
    \item For $k=n-1$, our sequence coincides with \cite[A000217]{OEIS} which counts many interesting combinatorial objects, including the number of edges in a complete graph of order $n+1$ as well as the number of legal ways to insert a pair of parentheses in a string of $n$ letters.
    \item For $k=n$, our sequence is the constant sequence $\left\{1, 1, \dots, \right\}$, which is straightforward as there is only one parking function with all points fixed, namely $\pi=(1, 2, \dots, n)$.
\end{itemize}

We have not found a match on OEIS for other more general $k$ values, but the above list already presents many intriguing research directions, which we hope to explore in future work.

Likewise, setting $k=0, \dots, n-1$ in Theorem \ref{prime-fixed} reveals mysterious coincidences that motivate further investigations.

\begin{itemize}
    \item For $k=0$, our sequence coincides with \cite[A007778]{OEIS} which is the number of ways of writing an $n$-cycle as the product of $n+1$ transpositions.
    \item For $k=1$, our sequence coincides with \cite[A055897]{OEIS} which is the total number of leaves in all labeled rooted trees with $n$ nodes.
    \item For $k=2$, our sequence coincides with \cite[A081132]{OEIS} which is the sum of all the fixed points in the set of endofunctions on $\{1, 2, \dots, n+1\}$.
    \item For $k=n-3$, our sequence coincides with \cite[A019582]{OEIS} which counts many interesting combinatorial objects, including half the number of colorings of $4$ points on a line with $n$ colors as well as the number of ways to place two dominoes horizontally in different rows on an $n \times n$ chessboard.
    \item For $k=n-2$, our sequence coincides with \cite[A002378]{OEIS} which are commonly referred to as the oblong numbers.
    \item For $k=n-1$, our sequence is the constant sequence $\left\{1, 1, \dots, \right\}$, which is straightforward as there is only one prime parking function with all points but one fixed, namely $\pi=(1, 2, \dots, n-1, 1)$.
\end{itemize}

\section*{Acknowledgements}
We thank Richard Stanley for the helpful discussions. Some of the results in this paper were independently conjectured by Ari Cruz. The authors are grateful for many valuable comments from the referees, which significantly improved the quality of the paper.

\end{document}